\documentclass[12pt,fleqn]{amsart}

\usepackage{macros}

\title{Hermitian rank in ideal powers}

\author[Abdullah Al Helal]{Abdullah Al Helal\,\orcidlink{0000-0002-4609-4511}}
\date{\today}
\address{Department of Mathematics, Oklahoma State University, Stillwater, OK 74078-5061}
\email{ahelal@okstate.edu}

\author[Ji{\v r}{\' i} Lebl]{Ji{\v r}{\' i} Lebl\,\orcidlink{0000-0002-9320-0823}}
\address{Department of Mathematics, Oklahoma State University, Stillwater, OK 74078-5061}
\email{lebl@okstate.edu}
\thanks{The second author was in part supported by Simons Foundation collaboration grant 710294.}


\begin{document}

\begin{abstract}
We prove that the (hermitian) rank of $QP^d$ is bounded from below by the rank of
$P^d$ whenever $Q$ is not identically zero and real-analytic in a neighborhood of some point on the zero set of $P$ in $\C^n$ and $P$ is a polynomial of bidegree at most $(1,1)$.
This result generalizes the theorem of D'Angelo and the
second author for bihomogeneous $P$ and Huang's Lemma.
The inequality demonstrates that real-analytic multipliers
cannot reduce the algebraic complexity of ideal powers.
Examples show that no hypothesis can be dropped.
As an application, we give a lower bound on the
target dimension for a CR map of hyperquadrics
depending on the order of contact.
\end{abstract}

\maketitle

\section{Introduction}
\label{sec:intro}

Artin's solution~\cite{artin-1927-uber} to 
Hilbert's 17th problem~\cite{hilbert-1900-mathematische}
says that every nonnegative polynomial on $\R^n$
is a sum
of squares of rational functions, providing an
algebraic proof of its nonnegativity and thus finding
applications in both pure and applied mathematics.
A sum of squares that is constantly 1 on a set $S$ 
induces a map of $S$ to the sphere
by considering the squared functions as components,
where the dimension is the number of squares.
Sums of squares thus find use in the
open problem of
classifying dimensions where a rational map of a sphere in 
$\R^n$ to a sphere in $\R^N$ exists,
a problem related to understanding the
homotopy groups of spheres,
see e.g.~\cite{bochnak-1998-real}*{chapter 13}.
Consequently, understanding
the number of squares used is vital.  
Pfister~\cite{pfister-1967-darstellung} proved
that at most $2^n$ squares are needed,
irrespective of the degree of the polynomial.

We are interested in the complex version of this circle of ideas.
Write a real polynomial $R$ on $\C^n \cong \R^{2n}$ as
a polynomial in $z=(z_1,\ldots,z_n)$ and its conjugate $\bar{z}$.
Such an $R$ can be written as a difference of squared norms
\begin{equation}
\label{eq:diff-norms}
    R(z,\bar{z})=\norm{F(z)}^2-\norm{G(z)}^2 ,
\end{equation}
where $\norm{\cdot}$ is the standard euclidean norm and
$F \colon \C^n \to \C^A$ and $G \colon \C^n \to \C^B$ are
vector-valued holomorphic polynomial maps.
Most interesting is the positive case, $B=0$.
Quillen~\cite{quillen-1968-representation} and later 
Catlin--D'Angelo~\cite{catlin-1996-stabilization}
(see also~\cite{dangelo-2019-hermitian}*{Section~4.12})
proved an analogue of Artin's result:
If a bihomogeneous polynomial $Q$, that is,
$Q(s z, \overline{t z}) = s^d\bar{t}^d Q(z,\bar{z})$,
is positive on the sphere, then there exists a $k$
and an $F$ such that
\begin{equation}
Q(z,\bar{z})\norm{z}^{2k}=\norm{F(z)}^2 .
\end{equation}
See D'Angelo~\cite{dangelo-2005-complex}
and Varolin~\cite{varolin-2008-geometry}
for
complex analogues of Hilbert's 17th problem.
D'Angelo and the second author~\cite{dangelo-2012-pfisters}
(see also~\cite{dangelo-2021-rational}*{Section~2.3})
proved that an analogue of Pfister's result does not hold:
The dimension $A$ goes to
infinity as $k$ tends to infinity.

We mostly forget about positivity and focus on the number $A+B$.
Define the \emph{rank} (sometimes called the \emph{hermitian rank})
of a polynomial $R(z,\bar{z})$ as the smallest $r$
with
\begin{equation}
    R(z,\bar{z}) = \sum_{k=1}^r \phi_k(z) \overline{\psi_k(z)}
\end{equation}
for some holomorphic polynomials $\phi_j$ and $\psi_k$.
The rank is the smallest number $A+B$
that can be used in the expansion~\cref{eq:diff-norms} if $R$ is real-valued,
although the definition we gave allows a complex-valued $R$.  The definition
extends to real-analytic functions $R$ provided we allow $r=\infty$.
The term ``rank'' is used as $r$ is the rank of
the matrix of coefficients of $R$.  See~\cref{sec:preliminaries} for more
detailed definitions.

Again, these ideas apply to understanding rigidity of rational sphere maps;
in this case rational maps taking the sphere in $\C^n$
to the sphere in $\C^N$.  In proving that all proper holomorphic maps of balls
that are $C^2$ up to the sphere for $n \leq N < 2n-1$
are equivalent to the linear embedding,
Huang~\cite{huang-1999-linearity} proved
the following useful lemma.  
For a real-analytic function $Q(z,\bar{z}) \not\equiv 0$ defined near the origin on $\C^n$,
\begin{equation} \label{equation-huang-lemma}
\rank Q(z,\bar{z}) \norm{z}^2 \geq n = \rank \norm{z}^2 .
\end{equation}
Huang's result still holds if we
replace the form $\norm{z}^2$ with an indefinite
form $\norm{z}_\ell^2 = -\abs{z_1}^2 - \cdots -\abs{z_\ell}^2 + \abs{z_{\ell+1}}^2
+ \cdots + \abs{z_n}^2$.
Xiao~\cite{xiao-2023-theorem} studied when equality holds in
\cref{equation-huang-lemma}
and proved that for $n \geq 3$, equality holds if and only if
$Q$ is of rank 1.  When $n=2$, there is a trivial counterexample to that statement:
$(\abs{z_1}^2-\abs{z_2}^2)(\abs{z_1}^2+\abs{z_2}^2) = \abs{z_1}^4-\abs{z_2}^4$.
Gao~\cite{gao-2025-exhaustion} further generalized Xiao's
result.

Let $(j,k)$ be the \emph{bidegree} of a polynomial $P(z, \bar{z})$ 
if the total degree in $z$ is $j$ and total degree in $\bar{z}$ is $k$.
For example, a sesquilinear form $P(z,\bar{z})$ is of bidegree $(1,1)$.
An arbitrary sphere or hyperquadric is precisely the zero set of a real-valued
bidegree $(1,1)$ polynomial.
If $P(0)=0$, an expansion of $QP$ into a difference of squares, or more appropriately 
to an equivalent form $\Re H(z) + \norm{F(z)}^2-\norm{G(z)}^2$ where
$H(0)=F(0)=G(0)=0$, is then equivalent to
finding holomorphic maps locally taking a sphere or a hyperquadric to a sphere
or a hyperquadric.

The aforementioned result of
D'Angelo and the second author~\cite{dangelo-2012-pfisters} is that
for any $d$,
\begin{equation} \label{equation-pfister-counterexample}
\rank Q(z,\bar{z}) \norm{z}^{2d} \geq \rank \norm{z}^{2d} .
\end{equation}
Our aim is to generalize \cref{equation-pfister-counterexample} to replace
$\norm{z}^{2}$ with an arbitrary polynomial of bidegree at most $(1,1)$.
That is, we not only allow arbitrary quadratic terms,
but also linear terms, and we allow the polynomial to be complex-valued.
We now state our main result.

\begin{theorem}[Main theorem]
\label[theorem]{thm:rank}
\pagebreak[2]
Let $n \geq 1$, $\mydeg \geq 0$,
$P(z,\bar{z})$ be a polynomial of bidegree at most $(1, 1)$ on $\C^n$ with a nonempty zero set,
and $Q(z,\bar{z}) \not\equiv 0$ be real-analytic and defined in a neighborhood of a point
$\mypt$ 
on the zero set of $P$.
Then
\begin{equation}
    \rank Q P^\mydeg \geq \rank P^\mydeg = \binom{\rank P + \mydeg - 1}{\mydeg}.
\end{equation}
\end{theorem}

\pagebreak[2]
The above is the most general statement of this result possible
in the sense that no hypothesis can be dropped, the bound is sharp,
and the power cannot be replaced by a product.
Firstly, the bound is sharp as equality occurs if $Q$ is a nonzero constant,
and it trivially fails if $Q\equiv 0$.

Less trivially, the conclusion fails if $P$ has no zero set:
For $\mydeg > 0$, $P = 1 + \norm{z}^2$, and $Q = \frac{1}{P^\mydeg}$,
we get that 
rank of $P^\mydeg$ is $\binom{n + \mydeg}{\mydeg} > 1$, but rank of $Q P^\mydeg$ is $1$.
More generally, the conclusion fails simply if $Q$ is not defined near any zero of $P$:
For $\mydeg > 0$, 
$P = 1 - \norm{z}^2$, and $Q = \frac{1}{P^\mydeg}$,
the rank of $P^\mydeg$ is $\binom{n + \mydeg}{\mydeg} > 1$, but rank of $Q P^\mydeg$
is $1$.

The conclusion may also fail if the bidegree of $P$ 
is bigger than $(1, 1)$:
For $\mydeg = 1$, $n = 2$, 
$P = \abs{z_1}^4 - \sqrt{2} \abs{z_1}^2 \abs{z_2}^2 + \abs{z_2}^4$, and 
$Q = \abs{z_1}^4 + \sqrt{2} \abs{z_1}^2 \abs{z_2}^2 + \abs{z_2}^4$,
we get that $P$ is of bidegree $(2, 2)$ and
rank of $P^\mydeg$ is $3$, but rank of $Q P^\mydeg = \abs{z_1}^8 + \abs{z_2}^8$ is $2$.
The proof of~\cite{dangelo-2011-hermitian-symmetric}*{Proposition~4.1} generalizes this
to a family of examples with $\mydeg$ a power of $2$.

Finally, the unrestricted product analogue 
$\rank Q P_1 \dots P_\mydeg \geq \rank P_1 \dots P_\mydeg$
fails even for distinct factors of the same bidegree:
For $\mydeg = 2$, $n = 2$, 
$P_1 = \abs{z_1}^2 - \exp(i\frac{\pi}{4}) \abs{z_2}^2$, 
$P_2 = \abs{z_1}^2 - \exp(-i\frac{\pi}{4}) \abs{z_2}^2$, and 
$Q = \abs{z_1}^4 + \sqrt{2} \abs{z_1}^2 \abs{z_2}^2 + \abs{z_2}^4$,
we get that each $P_k$ is of bidegree $(1, 1)$ and
rank of 
$P_1 \dots P_\mydeg
= \abs{z_1}^4 - \sqrt{2} \abs{z_1}^2 \abs{z_2}^2 + \abs{z_2}^4$ 
is $3$, but rank of 
$Q P_1 \dots P_\mydeg = \abs{z_1}^8 + \abs{z_2}^8$ is $2$.

The key idea in the proof is to reduce to the case when $P$ is of the form
\begin{equation}
    w + \bar{w} + \norm{z}^2 
    + \text{bidegree-$(1,1)$ terms involving $w$ or $\bar{w}$}
\end{equation}
where we split the variables to $z \in \C^{n-1}$ and $w \in \C$
and to move to a point when $Q$ is nonzero.
The combinatorics of the bound on the rank
in the case considered in \cite{dangelo-2012-pfisters} turns out to
be somewhat straightforward once the problem is viewed
in the correct context;
one bounds the rank of the matrix by considering the number of nonzero
entries on an extremal superdiagonal (or subdiagonal), and the count
reduces to what could be termed a ``monomial version'' of the problem.
In the presence of the linear terms $w + \bar{w}$, we can no longer
reduce to a single superdiagonal (a monomial version), and the
combinatorics required for the degree bound are significantly more difficult.
If $Q$ were a polynomial, then one could work in projective space
and get rid of the linear terms by an automorphism of ${\mathbb{P}}^n$.
However, if $Q$ is a real-analytic function, then such a change
of coordinates is unavailable.

We resolve this by applying a weighted quasi-homogeneous grading
inspired by Siegel domain structures.
The proof uses a weighted grading.
By defining topological weights for our variables, 
we extract the core of the polynomial via its weighted initial form,
bypassing combinatorics to guarantee full-rank structure.
This initial form or graded construction is 
standard in algebra~\cites{eisenbud-1995-commutative,sturmfels-1996-grobner}.

A key advantage of reliance on only the rank of the matrix
of coefficients is that this work can deal with the indefinite forms, unlike the previous work of \cites{quillen-1968-representation,catlin-1996-stabilization} requiring positivity.
This makes it a tool for CR geometry, 
not just positive-definite optimization.

Our result quantifies how badly the analog of Pfister's theorem fails.
See \cite{dangelo-2011-hermitian-symmetric} for more discussion of the
possible pairs $(A,B)$ that can arise in a product and
examples of the possible collapse of rank of a product.
In particular, the rank of a product can be 2, even if each of the
factors has arbitrarily large rank.
Powers of a real polynomial $p(x)$ in one variable may have fewer terms
than $p(x)$,
see e.g., Coppersmith--Davenport~\cite{coppersmith-1991-polynomials}.
Consider $P(z,\bar{z}) = p(\abs{z}^2)$, then
the hermitian rank of $P$ is precisely the number of terms
in $p$, and
we obtain examples where the hermitian rank of $P^d$ is lower than
the hermitian rank of $P$.

We have already mentioned the possible applications of this result to
polynomial optimization by writing a product as a sum of squares.
Bounding the rank of the product has potential algorithmic implications
in how efficiently such a polynomial problem can be solved.
We have also mentioned its relation to classification of maps between 
hyperquadrics. 
An application of our main result is \cref{thm:rigidity},
giving a relationship between order of contact of a map
and the target dimension.
Let us mention a couple of other connections.
The set of hermitian polynomials divisible by $P^\mydeg$
is the $d$-th power of the ideal $(P)$ in the hermitian polynomial ring.
Our result gives the lower bound 
$\binom{\rank P + \mydeg - 1}{\mydeg}$
on the hermitian rank of any element of $(P)^\mydeg$
that does not vanish to higher order,
and hence is a theorem in commutative algebra about ideal powers.
Our setup is also related to, although not the same, the 
Waring rank from algebraic geometry, that is, 
the minimum number of linear polynomials
so that the given polynomial can be written as the sum of their $d$th
powers.  See~\cite{carlini-2014-secant}.
Finally, 
a function $R(z,\bar{z})$ may be viewed as a reproducing
kernel on a Hilbert space (or Kre{\u\i}n space if
indefinite) of analytic functions.
The hermitian rank of $R$ corresponds to the dimension of this space.
So our result is a bound on the dimension of the
space for the kernel obtained by multiplying by a polynomial weight $P^d$.
See for example~\cite{agler-2002-pick} or \cite{gheondea-2015-reproducing}.

Before we move on, it may be good to 
contrast the real setting with the hermitian one.
A key difference
is how one finds the sum (or difference) of squares.
Writing a real polynomial or a real-analytic function
in $\C^n$ as a hermitian sum of squares
corresponds to diagonalization of the uniquely defined
corresponding hermitian matrix of coefficients
(see \cref{sec:preliminaries}).
For a real polynomial in $\R^n$,
one also diagonalizes a matrix, but
the symmetric matrix of coefficients
is not unique and hence finding a sum of squares 
representing a polynomial
is a convex optimization problem over Gram matrices,
which makes rank minimization a difficult semidefinite programming problem.
Notice for example that
\begin{equation}
    \begin{bmatrix} x^2 & xy & y^2 \end{bmatrix}
    \begin{bmatrix} 0 & 0 & c \\ 0 & -2c & 0 \\ c & 0 & 0 \end{bmatrix}
    \begin{bmatrix} x^2 \\ xy \\ y^2 \end{bmatrix}
    \equiv 0 \quad \text{for any } c .
\end{equation}
Thus, there is much less rigidity in writing the
sum of squares.
One example consequence of this difference is
the aforementioned theorem of Pfister,
and the lack of its analogue in the complex setting.
This distinction is
also visible on the fact that the dimensions of interest in the
real sphere mapping problem $F \colon S^n \to S^N$ are those
where $N$ is less than $n$, while in the complex case,
holomorphic maps taking the sphere $S^{2n-1} \subset \C^n$ to
$S^{2N-1} \subset \C^N$
are easily seen to be
constant if $N < n$, and the complex problem is interesting precisely
when $N > n$.


\section{Preliminaries}
\label{sec:preliminaries}

This section establishes the fundamental definitions and notations
for the (hermitian) rank of real-analytic functions.
These concepts form the basis of our result.
In this section, we write
$z = (z_1, \dots, z_n)$ and $\zeta = (\zeta_1, \dots, \zeta_n)$ 
for the coordinates in $\C^n$,
where $\zeta$ is used for polarization.

\begin{notation}
For any positive integer $k$,
denote by $[k]$ the set $\set{1, \dots, k}$ of integers from $1$ to $k$.
For convenience, we will let $[0]$ denote the empty set,
and $[\infty]$ denote the natural numbers $\N$.
We will use this notation extensively to index terms in sums.
\end{notation}

\begin{definition}
\label[definition]{def:rank}
Let $U \subset \C^n$ be a domain, and $R \colon U \to \C$ a
real-analytic function.
We define the
\emph{rank} of $R$ at
$p \in U$ to be 
the smallest $r \in \set{0} \cup \N \cup \set{\infty}$ such that
there exists a neighborhood $W$ of $p$ in $\C^n$ and 
holomorphic maps $\phi, \psi \colon W \to \C^r$ such that
\begin{equation}
\label{eq:rank}
    R(z, \bar{z}) = \sum_{k \in [r]} \phi_k(z) \overline{\psi_k(z)} .
\end{equation}
\end{definition}

It is easy to see that the rank is invariant under biholomorphic
changes of coordinates fixing $p$
by simply plugging the changes of coordinates into
the $\phi$ and the $\psi$.  It is slightly more complicated to show that
the rank does not depend on the point $p$.  Before we do so, we consider
the matrix of coefficients, which is an infinite matrix
that encodes the coefficients of the Taylor series of $R$
and shares the same rank as $R$.
One has to be careful about the domain of convergence of $R$ and 
the infinite matrix arising from $R$.
For more information on the ideas behind this approach in the real-valued setup,
see~\cites{dangelo-1993-several,grundmeier-2014-bounding,lebl-2020-hermitian}.
Let us develop the rank in the slightly more general complex-valued setting.
Without loss of generality, we may assume that $\mypt=0$ for simplicity.

\begin{remark}
If $R$ is real-valued, we can choose $\psi_k = \pm \phi_k$ for all $k \in [r]$,
and after a reordering, we get an $s$ with $0 \leq s \leq r$,
$\psi_k = \phi_k$ for all $1 \leq k \leq s$, and
$\psi_k = -\phi_k$ for all $s < k \leq r$.
Let $F = (\phi_k)_{1 \leq k \leq s} \colon W \subset \C^n \to \C^s$ and 
$G = (\phi_k)_{s < k \leq r} \colon W \subset \C^n \to \C^{r - s}$ to get
\begin{equation}
    R(z, \bar{z}) 
    = \sum_{1 \leq k \leq s} \abs{\phi_k(z)}^2 - \sum_{s < k \leq r}
    \abs{\phi_k(z)}^2
    = \norm{F(z)}^2 - \norm{G(z)}^2,
\end{equation}
which expresses $R$ as a difference of squared norms as in \cref{eq:diff-norms}.
\end{remark}

A key technique in our analysis is polarization of $R$, 
which refers to its extension to a function of $(z, \zeta)$
and allows us to treat holomorphic and anti-holomorphic parts separately.

\begin{notation}
A \emph{multi-index} $\ea = (\ea_1, \dots, \ea_n)$ 
is a vector of nonnegative integers in $(\set{0} \cup \N)^n$.
In multi-index notation, for $z = (z_1, \dots, z_n) \in \C^n$,
\begin{equation}
    z^\ea = z_1^{\ea_1} \dots z_n^{\ea_n}
\end{equation}
and
\begin{equation}
    \abs{\ea} = \sum_{k \in [n]} \ea_k = \ea_1 + \dots + \ea_n.
\end{equation}
\end{notation}

\begin{definition}
\label[definition]{def:matrix-coeff}
Fix an ordering on the monomials in $z$.
Let $\mathcal{Z} = (\dots, z^\ea, \dots) \colon \C^n \to \C^\infty$ 
be the map whose components are all monomials.
Suppose $R$ is a real-analytic function defined in a neighborhood
of the origin in $\C^n$.  Write
\begin{equation}    
    R(z, \bar{z}) = \sum R_{\ea\ec} z^\ea \bar{z}^\ec,
\end{equation}
where $\ea$ and $\ec$ are multi-indices.
The \emph{matrix of coefficients} of $R$ is the infinite matrix
\begin{equation}    
    \matcoeff{R} = [R_{\ea\ec}]
\end{equation}
whose columns are indexed by the holomorphic polynomials $z^\ea$
and rows by the antiholomorphic polynomials $\bar{z}^\eb$.
We will also call it the matrix of coefficients of the polarization 
$R(z, \zeta) = \sum R_{\ea\ec} z^\ea \zeta^\ec$.
\end{definition}

\textit{A priori}, the matrix of coefficients is a formal object.  However,
after a normalization, we will associate it to an operator on $\ell^2$.
Let us rescale and assume
that the polarized Taylor series for $R(z,\zeta)$ converges in a neighborhood of 
$\overline{\Delta \times \Delta} \subset \C^n \times \C^n$,
where $\Delta$ is the unit polydisc.
It is proved in~\cite{grundmeier-2014-bounding}*{Lemma~7}
that under this condition, $\matcoeff{R}$ defines a
compact trace-class operator on $\ell^2$.
Such operators are in one-to-one correspondence with the induced matrices,
and so we use $\matcoeff{R}$ for both the matrix and the
operator.
The key idea behind our bound for the rank of $\matcoeff{R}$ is
to bound the rank of some submatrix of the matrix representation of
$\matcoeff{R}$.

Using the vector $\mathcal{Z}$,
$R(z, \bar{z}) =  \inner{\matcoeff{R} \mathcal{Z} , \mathcal{Z}}$,
where $\inner{\cdot, \cdot}$ is the standard (hermitian) inner product.
Let $r \in \set{0} \cup \N \cup \set{\infty}$ denote the rank of 
the operator $\matcoeff{R}$.
The spectral theorem for trace-class operators and the singular
value decomposition 
(see, e.g., \cite{sunder-2015-operators} for trace-class operators and 
\cite{reed-1972-methods} for compact operators)
tell us that there are nonnegative numbers $\sigma_k \to 0$ for $k \in [r]$
called the singular values of $\matcoeff{R}$ and
corresponding orthonormal sets in $\ell^2$ with $r$ elements
$\set{u_1, u_2, \dots, u_r}$ and $\set{v_1, v_2, \dots, v_r}$,
called the left and right singular vectors respectively, such that
for any $x \in \ell^2$,
\begin{equation}
    \matcoeff{R} x =
    \sum_{k \in [r]} \sigma_k \inner{x,u_k} v_k .
\end{equation}
Let $\phi_k = \sqrt{\sigma_k} \inner{\mathcal{Z},u_k}$ and
$\psi_k = \sqrt{\sigma_k} \inner{\mathcal{Z},v_k}$ for $k \in [r]$.
As the vectors are square summable, we find that $\phi_k$
and $\psi_k$ are holomorphic functions of $\Delta$.
Then 
\begin{equation}
    R(z, \bar{z} )
    = \inner{\matcoeff{R} \mathcal{Z}, \mathcal{Z}}
    =
    \sum_{k \in [r]} \sigma_k \inner{\mathcal{Z},u_k} \inner{v_k,\mathcal{Z}}
    = \sum_{k \in [r]} \phi_k(z) \overline{\psi_k(z)} .
\end{equation}
This tells us that the decomposition in \cref{eq:rank} is always possible
and rank is a well-defined function.
The components of $\phi$ and $\psi$ are linearly independent.
If $R(z, \bar{z}) = \sum_{k \in [r']} \phi'_k(z) \overline{\psi'_k(z)}$
for some $r' \in \set{0} \cup \N \cup \set{\infty}$ and 
holomorphic maps $\phi', \psi' \colon W \subset \Delta \to \C^{r'}$, 
then reversing the process and possibly rescaling so that $W$ contains the
closed unit polydisc gives $\matcoeff{R}$ as a sum of $r'$
rank-$1$ operators.
So $r' \geq r$
and $r$ is the smallest such number required in the decomposition
in \cref{eq:rank}.
Therefore, $r$ is the rank of $R$ at $p$, and 
is the smallest such number in the decomposition in \cref{eq:rank}
if and only if the components of $\phi$ and $\psi$ are linearly independent.

\begin{proposition}
\label[proposition]{prop:rankindofpoint}
If $R \colon U \to \C$ is real-analytic and $U \subset \C^n$ is open and connected,
then the rank of $R$ is independent of the point $p \in U$.
\end{proposition}

\begin{proof}
A decomposition \cref{eq:rank} also holds at nearby points,
and so rank is upper-semicontinuous.
The coefficients of the matrix $\matcoeff{R}$ vary continuously
as the point $p$ moves.  If $\matcoeff{R}$ is of finite rank less
than or equal to $r$
on a sequence of points $q_k$ converging to $p$,
then all $(r+1)\times(r+1)$ subdeterminants of the matrix must be zero.
By continuity, they are also zero at $p$ and hence the rank is at most $r$
at $p$.  Hence the rank is lower-semicontinuous as well.  In other words,
the rank is constant.
\end{proof}

In fact, not only is the rank independent of biholomorphic coordinate
changes, we can, after polarization, change coordinates separately
in $z$ and $\zeta$.

\begin{proposition}
\label[proposition]{prop:polarbiholranksame}
Suppose that $U \subset \C^n$ is a domain,
$U^* = \{ z : \bar{z} \in U \}$ its conjugate,
and $R \colon U \times U^* \to \C$ a holomorphic function.
For two functions $a \colon V \to U$ and $b \colon V \to
U$ for a domain $V \subset \C^n$ that are biholomorphic onto their image,
\begin{equation}
    \rank R(z,\bar{z}) = \rank R\bigl( a(z), \overline{b(z)} \bigr) .
\end{equation}
\end{proposition}

\begin{proof}
Consider two points $p,q \in U$.  By connectedness,
there exists a continuous $c \colon [0,1] \to \C^n$
so that $p+c(0)=p$, $p+c(1)=q$ and a small polydisc centered at $p+c(t)$ is
a subset of $U$ for all $t \in [0,1]$.
That
$\rank R(z,\overline{z+c(t)})$
is independent of $t$ follows the same exact logic as the proof
of \cref{prop:rankindofpoint}.

Without loss of generality assume $0 \in V$.
By the argument above, we assume that $a(0)=b(0)=p \in U$.
There is a connected neighborhood $W$ of $p$
on which we have a
decomposition \cref{eq:rank}.  We may assume that $a$ and $b$
map into $W$ (possibly making $V$ smaller).
Then we simply plug $a$ and $b$ into \cref{eq:rank} and get 
$\rank R(z,\bar{z}) \geq \rank R\bigl( a(z), \overline{b(z)} \bigr)$,
and the opposite inequality follows by symmetry again.
\end{proof}

The proposition means that once we polarize, we can consider a point
that is not on the diagonal (the set where $\zeta = \bar{z}$) to
compute the rank.
Moreover, rank can be defined for the polarized 
$R$ and it is independent of the point in $U \times U^*$.

Note that $\matcoeff{R}$ depends on the ordering of the monomials, 
but its rank does not.
As the rank of $R$ coincides with the rank of $\matcoeff{R}$,
we will use rank of $R$ and rank of $\matcoeff{R}$ interchangeably.

\begin{definition}
By \emph{bidegree} of a polynomial $P(z, \bar{z})$, we mean a pair $(j, k)$
where $j$ is the total degree in $z$, and 
$k$ is the total degree in $\bar{z}$.
We will also call it the bidegree of the polarization $P(z, \zeta)$.
\end{definition}

For a polynomial $R$ of bidegree $(d,d)$, we only need to use
monomials up to degree $d$ in $\mathcal{Z}$ in the decomposition.
In this case, we will also assume that
$\matcoeff{R}$ is a finite matrix.

We will use the following very useful observation.

\begin{proposition}
\label[proposition]{prop:rank-Pd}
Let $n \geq 1$, $\mydeg \geq 0$, and $P(z,\bar{z})$ real-analytic near a point of $\C^n$.
Then
\begin{equation}
    \rank P^\mydeg \leq \binom{\rank P + \mydeg - 1}{\mydeg}.
\end{equation}
\end{proposition}

\begin{proof}
If $P(z, \bar{z}) = \sum_{k \in [\rank P]} \phi_k(z) \overline{\psi_k(z)}$
for holomorphic maps $\phi$ and $\psi$,
then the multinomial theorem decomposes $P^\mydeg$ as in \cref{eq:rank}
with $\binom{\rank P + \mydeg - 1}{\mydeg}$ terms.
In other words, $\rank P^\mydeg \leq \binom{\rank P + \mydeg - 1}{\mydeg}$.
\end{proof}


\section{Normalization}
\label{sec:normalization}

Let us make some reductions and prove several special cases
of our main result \cref{thm:rank}.
The first, and easiest, reduction is to assume after a translation
that $\mypt=0$ and therefore that $P(0)=0$.

We will normalize the polynomial $P$ and the real-analytic function $Q$.
In particular,
we will assume that $Q$ polarizes to some connected
neighborhood $U \times U^*$ of the origin in the polarized space.
From now on,
write $n' = n - 1 \geq 0$, and
$(z, w) = (z_1, \dots, z_{n'}, w) = (z_1, \dots, z_{n-1}, w)$ and
$(\zeta, \eta) = (\zeta_1, \dots, \zeta_{n'}, \eta) = (\zeta_1, \dots, \zeta_{n-1}, \eta)$ 
for the coordinates in $\C^n$,
where $(\zeta, \eta)$ is used for polarization.

By \cref{prop:polarbiholranksame}, the rank is independent of biholomorphic
changes of coordinates independently in $(z,w)$ and in $(\zeta,\eta)$,
as long as we do not leave $U \times U^*$, that is, as long
as the changes of coordinates take $(0,0)$ (in the polarized space) to
a near enough point.

We will assume that the monomials are in graded reverse lex order.
For $P$, it is sufficient to
consider only the columns corresponding to the monomials
$(1,z_1,\ldots,z_{n-1},w)$ in that order from left to right,
and rows corresponding to
$(1,\zeta_1,\ldots,\zeta_{n-1},\eta)$ from top to bottom.
Since $P(0)=0$,
then $\matcoeff{P}$, where only known terms are shown, is
\begin{equation}
    \matcoeff{P}
    =
    \begin{bmatrix}
        0 & & \\
          & &
    \end{bmatrix}.
\end{equation}

\begin{lemma}
\label[lemma]{lem:rank1or0}
\Cref{thm:rank} holds if the rank of $P$ is less than or equal to $1$,
or if $P$ is reducible.
\end{lemma}

\begin{proof}
If $\rank P = 0$, then $P \equiv 0$ and the conclusion
of \cref{thm:rank} holds trivially.
If $\rank P = 1$, then
$P = f \bar{g}$ for holomorphic (affine) $f$ and $g$,
so that $P^\mydeg = f^\mydeg \overline{g^\mydeg}$ is of rank $1$.
As $Q P^\mydeg \not \equiv 0$, $\rank Q P^\mydeg \geq 1 = \rank P^\mydeg$.

Assume $\mypt=0$.  If $P$ is reducible, write $P=P_1 P_2$ where 
$P_1$ and $P_2$ are affine linear and $P_1(0)=0$.
A linear change of coordinates in $(z,w)$ and
$(\zeta,\eta)$ independently is allowed,
and so 
after possibly swapping $(z,w)$ with $(\zeta,\eta)$,
assume that $P_1 = w$ or $P_1 = w+\eta$.
If $P_1 = w$, then $P_2$ has no terms in $z$ or $w$.
Otherwise, $P$
would have a bidegree $(2,0)$ term.
So $P_1$ is a function of $(z,w)$ and $P_2$ is a function of
$(\zeta,\eta)$ and hence $P$ is of rank $1$.
If $P_1 = w+\eta$, then $P_2$ has no linear terms, otherwise
$P$ would have a bidegree $(2,0)$ or $(0,2)$ term.
So $P_2$ is constant and $P$ is not reducible.
In other words, if $P$ is reducible, then its rank is 1.
\end{proof}

\begin{lemma}
\label[lemma]{lem:normalQnotzero}
    To prove \cref{thm:rank}, it is sufficient to prove the
conclusion of the theorem when $Q(\mypt,\bar{\mypt})\neq 0$.
\end{lemma}

\begin{proof}
First, we can assume that $\mypt=0$.
If the polarized $Q$ is nonzero at some point on the zero set
of the polarized $P$, then it is nonzero on a point arbitrarily close to
the origin in the polarized space and hence we
can apply a small translation (which can be different in the
$(z,w)$ and in the $(\zeta,\eta)$ variables)
and via
\cref{prop:polarbiholranksame} work at a point where $Q$ is not equal
to $0$.

Otherwise, the polarized $Q$ is zero on the zero set of the polarized $P$.
By \cref{lem:rank1or0}, we can assume that $P$ is irreducible, and 
so after another possible translation we can also assume that the derivative
of $P$ does not vanish at the origin.
Therefore,
$P$ generates the ideal of germs at the origin of holomorphic functions (in the polarized space)
vanishing on the zero set of $P$.
As the zero set of $P$ is contained on the zero set of $Q$,
$P$ divides $Q$ in the ring of germs of holomorphic functions at the origin.
Thus, there is a real-analytic function $Q'$ 
in a neighborhood of the origin and a positive integer $\mydeg'$ such that
$Q = Q' P^{\mydeg'}$
and $Q'$ is not identically zero on the zero set of $P$.
Assuming the result holds for $P$ and $Q'$,
we get $\rank P^\mydeg = \binom{\rank P + \mydeg - 1}{\mydeg}$ and
\begin{equation}
\begin{split}
    \rank Q P^\mydeg 
    &= \rank Q' P^{\mydeg'+\mydeg} \\
    &\geq \rank P^{\mydeg'+\mydeg} \\
    &= \binom{\rank P + \mydeg' + \mydeg - 1}{\mydeg' + \mydeg} \\
    &= \binom{\rank P + \mydeg' + \mydeg - 1}{\rank P - 1} \\
    &\geq \binom{\rank P + \mydeg - 1}{\rank P - 1} \\
    &= \binom{\rank P + \mydeg - 1}{\mydeg} \\
    &= \rank P^\mydeg.
\end{split}
\end{equation}
The result holds.
\end{proof}

\begin{lemma}
\label[lemma]{lem:normallinear}
Let $P$ be a polynomial of bidegree at most $(1,1)$ such that $P(0)=0$
and polarized as above.
Then up to applying a linear change of coordinates in $(z,w)$ variables
and an independent linear change of coordinates in the $(\zeta,\eta)$
variables, and possibly swapping $(z,w)$ with $(\zeta,\eta)$,
we can assume that $P$ is of
one of three forms:
\begin{enumerate}
\item
$\displaystyle
    P   
    = w + \eta + \sum_{k \in [r]} z_k \zeta_k
        + \text{bidegree-$(1,1)$ terms involving $w$ or $\eta$} ,
$
\item
$\displaystyle
    P   
    = w + \sum_{k \in [r]} z_k \zeta_k
        + \text{bidegree-$(1,1)$ terms involving $w$ or $\eta$} ,
$
\item
$\displaystyle
    P   = \sum_{k \in [r]} z_k \zeta_k + w \eta .
$
\end{enumerate}
\end{lemma}

\begin{proof}
Arbitrary linear changes of coordinates may be
made in $(z,w)$ and $(\zeta,\eta)$ independently.
If $P(z, w, \zeta, \eta)$ has a nonzero linear term in $(z, w)$, 
we can transform it into only $w$ 
by a linear transformation followed by a nonzero scaling in coordinates $(z, w)$
without affecting the constant.
Similarly,
if $P(z, w, \zeta, \eta)$ has a nonzero linear term in $(\zeta, \eta)$, 
we can transform it into only $\eta$ 
by a linear transformation followed by a nonzero scaling in coordinates $(\zeta, \eta)$
without affecting the constant.
We obtain the following $\matcoeff{P}$,
where only known terms are shown:
\begin{equation}
    \matcoeff{P}
    =
    \begin{bmatrix}
        0 & 0 & \dots & 0 & (0\text{ or }1) \\
        0 &   &       &   & \\
        \vdots & &    &   & \\
        0 &   &       &   & \\
        (0\text{ or }1) & & & & \\
    \end{bmatrix}.
\end{equation}
The linear term of $P(z, w, \zeta, \eta)$ becomes $w + \eta$, $w$, $\eta$, or $0$.
If the linear term is $\eta$, we can swap $(z,w)$ with $(\zeta,\eta)$.

If the linear term of $P(z, w, \zeta, \eta)$ is $w + \eta$ or $w$,
we consider the submatrix of $\matcoeff{P}$ obtained by
disregarding the row corresponding to $\eta$ and the
column corresponding to $w$, that is, the
submatrix corresponding to products of $z$ with $\zeta$.
Linear transformations in the $z$ and $\zeta$ independently
transform this 
submatrix into a matrix with $1$s and zeros on the diagonal.
That is, either
\begin{equation}
\label{eq:matrix-form-original-1}
    \matcoeff{P}
    =
    \begin{bmatrix}
        0   & 0   & 0 & 1 \\
        0   & I_r & 0 & \\
        0   & 0   & 0 & \\
        1   &     &   &
    \end{bmatrix}
\end{equation}
and
\begin{equation}
\label{eq:normal-form-original-1}
\begin{split}  
    P   
    &= w + \eta + \sum_{k \in [r]} z_k \zeta_k
        + \text{bidegree-$(1,1)$ terms involving $w$ or $\eta$} \\
    &= w \Bigl(1 + \sum_{k \in [n']} P_{\zeta_k} \zeta_k + P_\eta \eta \Bigr)
    + \sum_{k \in [r]} z_k (\zeta_k + P_{z_k} \eta)
    + \sum_{k > r} z_k (P_{z_k} \eta)
    + \eta,
\end{split}
\end{equation}
or
\begin{equation}
\label{eq:matrix-form-original-2}
    \matcoeff{P}
    =
    \begin{bmatrix}
        0   & 0   & 0 & 1 \\
        0   & I_r & 0 & \\
        0   & 0   & 0 & \\
        0   &     &   &
    \end{bmatrix}
\end{equation}
and
\begin{equation}
\label{eq:normal-form-original-2}
\begin{split}
    P   
    &= w + \sum_{k \in [r]} z_k \zeta_k
        + \text{bidegree-$(1,1)$ terms involving $w$ or $\eta$} \\
    &= w \Bigl(1 + \sum_{k \in [n']} P_{\zeta_k} \zeta_k + P_\eta \eta \Bigr)
    + \sum_{k \in [r]} z_k (\zeta_k + P_{z_k} \eta)   
    + \sum_{k > r} z_k (P_{z_k} \eta)
\end{split}
\end{equation}
for some $0 \leq r \leq n'$ and
some $P_{\zeta_k}$, $P_{z_k}$, and $P_\eta$ for $k \in [n']$.
Here, $r = 0$ means that the sum over $k \in [r]$ is vacuous
and the rows and columns of $\matcoeff{P}$ corresponding to $I_r$ are nonexistent.

If the linear term of $P(z, w, \zeta, \eta)$ is $0$,
we consider the remaining submatrix corresponding to products of $(z, w)$ with $(\zeta, \eta)$.
Making independent linear transformations in $(z,w)$ and $(\zeta,\eta)$
we can transform
this submatrix
into a diagonal matrix with $1$s and $0$s on the diagonal.
We can make sure that the term corresponding to $w\eta$ is $1$ (as the
matrix is not the zero matrix) and
we order the $1$s to come first in the $z$ and $\zeta$ coordinates
as before.
That is,
\begin{equation}
\label{eq:matrix-form-original-3}
\matcoeff{P} = 
    \begin{bmatrix}
        0 & 0   & 0 & 0 \\
        0 & I_r & 0 & 0 \\
        0 & 0   & 0 & 0 \\
        0 & 0   & 0 & 1
    \end{bmatrix}
\end{equation}
and
\begin{equation}
\label{eq:normal-form-original-3}
    P   = \sum_{k \in [r]} z_k \zeta_k + w \eta
\end{equation}
for some $0 \leq r \leq n'$.
\end{proof}

\begin{lemma}
\label[lemma]{lem:fullranknormform}
    To prove \cref{thm:rank}, it is sufficient to prove the
conclusion of the theorem when $\mypt=0$, $Q$ is nonzero at the origin,
and the polarized $P$ is
\begin{equation}
\label{eq:normal-form}
    P   = w + \eta + z \cdot \zeta
        + \text{bidegree-$(1,1)$ terms involving $w$ or $\eta$},
\end{equation}
where $\cdot$ is the standard bilinear product.
\end{lemma}

We remark that if $n=1$, then we mean that
$P=w+\eta+C w\eta$ for some constant $C$.

\begin{proof}
As we said, it is sufficient to assume that $\mypt=0$.
We can also assume that rank of $P$ is at least $2$ via \cref{lem:rank1or0},
and we can assume that $Q$ is not zero at the origin.
We will work in the polarized setting as before and
treat $(z,w)$ and $(\zeta,\eta)$ as independent.

By the previous lemma, it is sufficient for $\matcoeff{P}$ to
be of one of three different forms
\cref{eq:matrix-form-original-1,eq:matrix-form-original-2,eq:matrix-form-original-3}.
The matrix $\matcoeff{P}$ is of full rank, that is, of rank $r + 2 = n + 1 = n' + 2$ 
if and only if it is of the form \cref{eq:matrix-form-original-1} and $r = n'$, that is, 
it is of the form
\begin{equation}
\label{eq:matrix-form}
    \matcoeff{P}
    =
    \begin{bmatrix}
        0 & 0      & 1 \\
        0 & I_{n'} &  \\
        1 &        &
    \end{bmatrix}
\end{equation}
and
\begin{equation}
    P   = w + \eta + z \cdot \zeta
        + \text{bidegree-$(1,1)$ terms involving $w$ or $\eta$}.
\end{equation}
Assume that the result holds for this form.

First, assume $\matcoeff{P}$ is of the form
\cref{eq:matrix-form-original-1} but not of full rank;
then
$2 \geq \rank P = r + 2$ and $r \geq 0$.
We let $P'$ and $Q'$ be the polynomial and the real-analytic function obtained from 
setting $z_{r+1}, \dots, z_{n'}$ and $\zeta_{r+1}, \dots, \zeta_{n'}$ to zero 
in the Taylor series of $P$ and $Q$ respectively to find that
\begin{equation}
    P' 
    = w \Bigl(1 + \sum_{k \in [r]} P_{\zeta_k} \zeta_k + P_\eta \eta \Bigr)
    + \sum_{k \in [r]} z_k (\zeta_k + P_{z_k} \eta)
    + \eta.
\end{equation}
Note that $P'(z_1, \dots, z_r, w, \zeta_1, \dots, \zeta_r, \eta)$ and 
$Q'(z_1, \dots, z_r, w, \zeta_1, \dots, \zeta_r, \eta)$ satisfy the hypotheses of the result under discussion with $n' = r$, and
$\rank P' = \rank P = r + 2$.
So $P'$ is of the form \cref{eq:normal-form}, and thus the result follows.

Next, assume $\matcoeff{P}$ is of the form \cref{eq:matrix-form-original-2}.
Since $2 \geq \rank P = r + 1$, we get $r \geq 1$.
Let $\epsilon > 0$ be small enough so that the point
where all (polarized) variables are zero except $z_r = \epsilon$
is still within the domain of convergence of the polarized $Q$
and such that $Q$ is not zero at this point.
By \cref{prop:polarbiholranksame}, we can move to this point.
That is,
we change variables by replacing $z_r$ with $z_r+\epsilon$
and swapping $\eta$ and $\zeta_r$.  In these new coordinates,
$P$ has a linear term in $w$ and a linear term in $\eta$,
and still vanishes at the origin.
We can now apply the normalization of \cref{lem:normallinear}
and we find that $P$ is of the form \cref{eq:matrix-form-original-1}
that we already handled.

Finally, assume $\matcoeff{P}$ is of the form \cref{eq:matrix-form-original-3}.
Since $2 \geq \rank P = r + 1$, we get $r \geq 1$.
Again let 
$\epsilon > 0$ be small enough so that the point
where all (polarized) variables being zero except $z_r = \epsilon$ and $w = \epsilon$
is still within the domain of convergence of the polarized $Q$
and such that $Q$ is nonzero there.
By \cref{prop:polarbiholranksame}, we can move to this point,
and this time we 
change variables by replacing $z_r$ with $z_r+\epsilon$
and $w$ with $w+\epsilon$,
and swapping $\eta$ and $\zeta_r$.  Again, this creates a
linear term in both $w$ and $\eta$ and after normalization by
\cref{lem:normallinear}, we reduce to the form
\cref{eq:matrix-form-original-1}.
\end{proof}


\section{Weighted Initial Form}

We now prove the full-rank case.
The argument is diagonalization using a weighted quasi-homogeneous grading.
Weighted initial forms and associated grading
construction are standard in commutative algebra and 
Gröbner theory~\cites{eisenbud-1995-commutative,sturmfels-1996-grobner}.
Weighted models are also standard in several complex variables and 
CR geometry~\cites{catlin-1984-boundary,kolar-2014-chern}.

Assign positive weights to the polarized variables.
The weighted initial form $\initial(R)$
of an arbitrary nonzero convergent Taylor series $R$
is the sum of all monomials of $R$ with the smallest weight.
As all weights are positive, $\initial(R)$ is a polynomial.

\begin{lemma}
\label[lemma]{lem:initial}
    Let $R$ be a nonzero real-analytic function near the origin,
    polarized in independent variables.
    Then 
    \begin{equation}
    \label{eq:initial-rank}
        \rank R \geq \rank \initial(R).
    \end{equation}
\end{lemma}

\begin{proof}
    Suppose the two polarized variable blocks are
    $z = (z_1, \dots, z_n)$ and $\zeta = (\zeta_1, \dots, \zeta_n)$
    with weights $a_1, \dots, a_n$ and $b_1, \dots, b_n$ respectively.
    Let $m$ be the smallest weight in $R$ and
    for $0 < t < 1$, set
    \begin{equation}
        R_t(z, \zeta) = t^{-m} R(t^{a_1} z_1, \dots, t^{a_n} z_n,
                                 t^{b_1} \zeta_1, \dots, t^{b_n} \zeta_n).
    \end{equation}

    As $\matcoeff{R_t}$ is obtained by multiplying rows and columns of
    $\matcoeff{R}$ by nonzero scalars, followed by scaling by $t^{-m}$,
    \cref{prop:polarbiholranksame} gives us
    \begin{equation}
        \rank R = \rank R_t.
    \end{equation}
    If $c_{\alpha\beta}$ is the coefficient of $z^\alpha \zeta^\beta$ in $R$,
    then its coefficient in $R_t$ is 
    $c_{\alpha\beta} t^{a \cdot \alpha + b \cdot \beta - m}
    = c_{\alpha\beta} t^{\weight(z^\alpha \zeta^\beta) - m}$.
    The exponent is zero for the monomials in $\initial(R)$ and 
    positive for the other monomials.
    Consequently, $R_t \to \initial(R)$ coefficientwise as $t \to 0^+$.

    Put $r = \rank \initial(R) > 0$.
    Take any $r \times r$ submatrix of $\matcoeff{\initial(R)}$ with nonzero determinant.
    The determinant of the corresponding submatrix of $\matcoeff{R_t}$ converges to that nonzero determinant, and 
    so it remains nonzero for all small enough $t > 0$ by continuity.
    Therefore, $\rank R_t \geq r$ for such $t$, and hence $\rank R \geq r$.
\end{proof}

We have the immediate identity
\begin{equation}
\label{eq:initial-product}
    \initial(R_1 R_2) = \initial(R_1) \initial(R_2)
\end{equation}
for two nonzero convergent Taylor series $R_1$ and $R_2$,
as the smallest weights add.

\begin{proof}[Proof of \cref{thm:rank}]
    By \cref{lem:fullranknormform}, we can assume that
    $\mypt=0$, $Q$ is nonzero at the origin, and the polarized $P$ is
    $P   = w + \eta + z \cdot \zeta
        + \text{bidegree-$(1,1)$ terms involving $w$ or $\eta$}$.
    Inspired by the Siegel domain structure given by
    $w + \bar{w} - \norm{z}^2 > 0$ or equivalently its polarization
    $w + \eta - z \cdot \zeta > 0$,
    we assign weights
    \begin{equation}
        \weight z_k = \weight \zeta_k = 1, \quad
        \weight w = \weight \eta = 2.
    \end{equation}
    Set
    \begin{equation}
        H = w + \eta + z \cdot \zeta
    \end{equation}
    and notice that every monomial in $H$ has a weight of $2$,
    but the bidegree-$(1,1)$ terms involving $w$ or $\eta$
    can have monomial types of $w \zeta_k$, $z_k \eta$ and $w \eta$,
    with weights $3$, $3$ and $4$, respectively.
    This gives us
    $\initial(P) = H$.
    On the other hand, $Q(0) \neq 0$ has a weight of zero, and 
    all variable weights are positive, so 
    $\initial(Q) = Q(0)$.
    Therefore
    \begin{equation}
        \initial (QP^\mydeg)
        = \initial (Q) \initial(P)^\mydeg
        = Q(0) H^\mydeg.
    \end{equation}
    It remains to find the rank of $H^\mydeg$.
    The multinomial theorem tells us that
    \begin{equation}
    \label{eq:initial-power}
        H^\mydeg 
        = (w + \eta + z \cdot \zeta)^\mydeg
        = \sum_{\ell + m + \abs{\gamma} = \mydeg} 
        \binom{\mydeg}{\ell ; m ; \gamma} w^\ell \eta^m z^\gamma \zeta^\gamma.
    \end{equation}

    As $P$ is of bidegree $(1, 1)$, 
    $P^\mydeg$ is of bidegree $(\mydeg, \mydeg)$, and so
    the monomials in $P^\mydeg$ are given by
    $\set{z^\ea w^\eb \zeta^\ec \eta^\ed \suchthat 
    0 \leq \abs{\ea} + \eb \leq d, 0 \leq \abs{\ec} + \ed \leq d}$.
    The holomorphic monomials 
    $z^\ea w^\eb$ with $0 \leq \abs{\ea} + \eb \leq d$
    are in bijection with the antiholomorphic monomials 
    $\zeta^\ea \eta^{d - (\abs{\ea} + \eb)}$.
    Once we fix the holomorphic monomial $z^\ea w^\eb$,
    $\zeta^\ea \eta^{d - (\abs{\ea} + \eb)}$ is the only antiholomorphic
    monomial occuring in $H^\mydeg$ due to \cref{eq:initial-power}.
    The coefficient of the corresponding monomial
    $z^\ea w^\eb \zeta^\ea \eta^{d - (\abs{\ea} + \eb)}$ is
    \begin{equation}
    \binom{\mydeg}{\eb ; d - (\abs{\ea} + \eb) ; \ea}
    = \frac{d!}{\eb! (d - (\abs{\ea} + \eb))! \ea_1! \dots \ea_{n'}!} 
    > 0.
    \end{equation}
    That is, after a permutation of rows, the matrix $\matcoeff{H^\mydeg}$
    is diagonal with nonzero entries.
    \cref{prop:rank-Pd} and \cref{eq:initial-rank} give us
    \begin{equation}
    \begin{split}
        \rank P^\mydeg
        &\geq \rank H^\mydeg \\
        &= \rank \matcoeff{H^\mydeg} \\
        &= \# \set{z^{\ea} w^{\eb} \suchthat 0 \leq \abs{\ea} + \eb \leq \mydeg} \\
        &= \binom{n + \mydeg}{\mydeg} \\
        &= \binom{\rank P + \mydeg - 1}{\mydeg} \\
        &\geq \rank P^\mydeg,
    \end{split}
    \end{equation}
    which makes $\matcoeff{P^\mydeg}$ full rank with 
    $\rank H^\mydeg = \rank P^\mydeg = \binom{\rank P + \mydeg - 1}{\mydeg}$.
    By another application of \cref{eq:initial-rank},
    \begin{equation}
    \begin{split}
        \rank Q P^\mydeg
        &\geq \rank \initial (Q P^\mydeg) \\
        &= \rank (Q(0) \initial(P)^\mydeg) \\
        &= \rank H^\mydeg \\
        &= \rank P^\mydeg \\
        &= \binom{\rank P + \mydeg - 1}{\mydeg},
    \end{split}
    \end{equation}
    which proves the main theorem.
\end{proof}


\section{Application to CR Maps}
\label{sec:application}

As we discussed in \cref{sec:intro}, a hyperquadric
$M \subset \C^n$ is
precisely the zero set of a real-valued bidegree $(1,1)$ polynomial
$P$ of maximal rank at any point of $M$, that is, of rank $n+1$.
Our result gives a bound on the target dimension for maps between
hyperquadrics with prescribed order of contact.
By order of contact we mean the following.
Suppose a manifold $M$ has defining function $\rho$ and $M'$ has defining
function $\rho'$.  We will say that a mapping $f$ defined in a neighborhood
of $M$ and taking $M$ to $M'$ has order of contact $k$ if
$\rho' \circ f = O(\abs{\rho}^k)$ but $\rho' \circ f$ is not
$o(\abs{\rho}^k)$.  When $f$ is holomorphic and $\rho$ and $\rho'$
are real-analytic, then this is equivalent to $k$ being the largest
$k$ so that $\rho' \circ f = q \rho^k$.
Although rank can collapse drastically under multiplication in general, 
referred to as \emph{collapsing of rank} in~\cite{dangelo-2011-hermitian-symmetric}, 
our main theorem says that
multiplication by powers of a hyperquadric defining function cannot 
reduce the rank below the rank of the defining power,
provided the multiplier is genuinely analytic in a neighborhood of some
point of the hyperquadric.
This naturally leads to the following theorem showing that
high order of contact is only possible for large target dimension.

\begin{theorem}
\label[theorem]{thm:rigidity}
    Let $M \subset \C^n$ be a hyperquadric or a sphere,
    $U \subset M$ an open set,
    $M' \subset \C^N$ another hyperquadric or a sphere,
    $F \colon U \to M' \subset \C^N$ a real-analytic CR map, 
    and $\mydeg$ the order of contact of $F$.
    Then $N \geq \binom{n + \mydeg}{\mydeg}-1$.
\end{theorem}

We remark that if $M$ is not a sphere, then we
need not assume that $F$ is CR, as $F$ will automatically extend
to a neighborhood by Lewy's extension theorem.

\begin{proof}
Let $P$ be a bidegree $(1,1)$ defining polynomial for $M$.
As $P$ is a hyperquadric defining function, it has $\rank P = n+1$.
Let $P'$ be the defining polynomial for $M'$, and here
$\rank P' = N+1$.
    We work in a neighborhood of some point $p \in U$.
    Since $F$ is a real-analytic CR map,
    it can be holomorphically extended to a neighborhood of $p$ in $\C^n$.
    As this extension takes $M$ to $M'$, 
    we have $P'(F(z), \overline{F(z)}) = 0$ whenever $P(z, \bar{z}) = 0$.
    Thus there is a real-analytic function $Q(z, \bar{z}) \not\equiv 0$ 
    defined on a neighborhood of $p$
    and a maximal integer $\mydeg$ such that 
    \begin{equation}
        P'(F(z), \overline{F(z)}) = Q(z, \bar{z}) P(z, \bar{z})^\mydeg.
    \end{equation}
    The maximality of $\mydeg$ makes it the order of contact.
    By our main theorem,
    \begin{equation}
        \rank P'(F(z), \overline{F(z)})
        = \rank Q(z, \bar{z}) P(z, \bar{z})^\mydeg
        \geq \rank P(z, \bar{z})^\mydeg
        = \binom{\rank P + \mydeg - 1}{\mydeg}.
    \end{equation}
    If $P'(z, \bar{z}) = \sum_{k \in [\rank P']} \phi_k(z) \overline{\psi_k(z)}$
    for holomorphic maps $\phi$ and $\psi$,
    then 
    $P'(F(z), \overline{F(z)}) = \sum_{k \in [\rank P']} \phi_k(F(z)) \overline{\psi_k(F(z))}$,
    so $\rank (P' \circ F) \leq \rank P'$.
    That is, precomposition with a map cannot increase rank.
    
    We now get a restriction on the target dimension:
    \begin{equation}
        N+1
        = \rank P'(z, \bar{z})
        \geq \rank P'(F(z), \overline{F(z)})
        \geq \binom{\rank P + \mydeg - 1}{\mydeg}
        = \binom{n + \mydeg}{\mydeg}.
    \qedhere
    \end{equation}
\end{proof}

\begin{remark}
    A high-order forces a high rank with polynomial growth, as asymptotically
    \begin{equation}
        N+1 \geq \binom{\mydeg + n}{\mydeg} \sim \frac{\mydeg^{n}}{n!}.
    \end{equation}
\end{remark}


\bibliographystyle{amsplain}
\bibliography{main}

\end{document}